%% file: alloc11.tex
\documentclass[12pt]{amsart}
\usepackage{amsmath,amssymb,amsfonts,amsthm,amstext,bbm}
\usepackage{color,graphicx}

\newcommand{\Q}{{\mathbb Q}}
\newcommand{\Z}{{\mathbb Z}}
\newcommand{\R}{{\mathbb R}}

\renewcommand{\P}{{\mathbb P}}
\newcommand{\E}{{\mathbb E}}

\newcommand{\df}{\textbf}

\newcommand{\leb}{\mathcal{L}}

\newcommand{\union}{\cup}
\newcommand{\Union}{\bigcup}
\newcommand{\intersect}{\cap}

\newtheorem{thm}{Theorem}
\newtheorem{lemma}[thm]{Lemma}

\newcounter{mycount}

\title{Poisson allocations with bounded connected cells}

\author{Alexander E. Holroyd}
\address{Alexander E. Holroyd, Microsoft Research, 1 Microsoft Way, Redmond, WA 98052, USA}
\email{holroyd at microsoft.com}

\author{James B. Martin}
\address{James B. Martin, Department of Statistics,
1 South Parks Rd, Oxford OX1 3TG, UK}
\email{martin at stats.ox.ac.uk}

\date{8 October 2014}

\begin{document}

\begin{abstract}\sloppypar
Given a homogenous Poisson point process in the plane, we prove that it is
possible to partition the plane into bounded connected cells of equal volume,
in a translation-invariant way, with each point of the process contained in
exactly one cell.  Moreover, the diameter $D$ of the cell containing the
origin satisfies the essentially optimal tail bound $\P(D>r)<c/r$.  We give
two variants of the construction.  The first has the curious property that
any two cells are at positive distance from each other.  In the second, any
bounded region of the plane intersects only finitely many cells almost
surely.
\end{abstract}

\maketitle

\section{Introduction}
Let $\Pi$ be a simple point process on $\R^d$. Its \df{support} is the random
set of points $[\Pi]:=\{x\in\R^d:\Pi(\{x\})=1\}$. Let $\leb$ denote Lebesgue
measure or volume on $\R^d$.  An \df{allocation} of $\Pi$ (to $\R^2$) is a
random measurable map $\Phi:\R^d\to\R^d \cup\{\infty\}$ such that almost
surely $\Phi(x)\in[\Pi]$ for $\leb$-almost every $x\in\R^d$, and $\Phi(x)=x$
for all $x\in[\Pi]$.  For a point $x\in[\Pi]$, the set $\Phi^{-1}(x)$ is
called the \df{cell} of $x$.  (The reason for allowing a null set to be
mapped to $\infty$ is to avoid uninteresting complications concerning
boundaries of cells.) An allocation $\Phi$ is \df{translation-invariant} if
for every $y\in\R^d$, the map $x\mapsto \Phi(x-y)+y$ has the same law as
$\Phi$.

Of particular interest are translation-invariant \df{fair}
allocations, in which all cells have equal volume.  Such
allocations were introduced in \cite{hp} as a tool in the
construction of shift-couplings of Palm processes. Several
specific choices of allocation have been studied in depth
\cite{grav,grav2,hhp,hhp2,hpps,krikun,timar-marko,nazarov-sodin-volberg}.
A particular focus is on bounding the diameter of a typical
cell, for allocations to a homogenous Poisson point
process.

In the plane $\R^2$, it is natural to ask whether all cells
can be connected sets.  (This is clearly impossible in
$\R$, while in $\R^d$ for $d\geq 3$ it is straightforward
to modify any allocation to make the cells connected).
Krikun \cite{krikun} constructed the first
translation-invariant fair allocation of a Poisson process
to $\R^2$ with connected cells (answering a question in
\cite{hp}), but was unable to determine whether its cells
are bounded.  Here we construct an allocation whose cells
are both connected and bounded, answering a question posed
by Scott Sheffield and Yuval Peres (personal
communications).

\begin{thm}\label{alloc}
Let $\Pi$ be a homogeneous Poisson point process of
intensity $1$ on $\R^2$. There exists a
translation-invariant allocation of $\Pi$ in which almost
surely each cell is a bounded, connected set of area $1$
that contains the allocated point. Moreover, the diameter
$D$ of the cell containing the origin satisfies
$\P(D>t)<c/t$ for some $c$ and all $t>0$, and in addition
we may choose either one of the following properties:
\begin{itemize}
\item[(a)] any two cells are at non-zero distance from each other; or
\item[(b)] any bounded set in $\R^2$ intersects only finitely many cells.
\end{itemize}
\end{thm}

It is easily seen that no allocation can satisfy both (a)
and (b): (a) implies that the line segment joining any two
points of $[\Pi]$ intersects infinitely many cells, in
contradiction to (b).  In the above, the \df{diameter} of a
set $A\subseteq\R^2$ is $\sup_{x,y\in A} \|x-y\|$, where
$\|\cdot\|$ denotes the Euclidean norm.  The power $-1$ of
$t$ in the tail bound cannot be improved: any
translation-invariant fair allocation of a homogenous
Poisson process satisfies $\E D = \infty$; see \cite{hp}.

In contrast with the allocations considered in
\cite{grav,hhp,krikun}, those that we provide are not
especially canonical.  Rather, the point is that, armed
with appropriate tools, it is not difficult to construct
allocations with a variety of desirable properties.  The
two parts (a) and (b) will use similar constructions, with
the first being slightly simpler.  Our allocations are not
deterministic functions of the point process $\Pi$, but
require additional randomness. See e.g.\ \cite{hpps} for
more on this distinction (especially in the context of
matchings). It remains an open question to prove the
existence of a translation-invariant fair allocation with
bounded connected cells in $\R^2$ that is a deterministic
function of the Poisson process.  It is plausible this
could be done by combining our methods with deterministic
hierarchical partitioning techniques as in e.g.\
\cite{hp-tree-match,soo,timar}.

\section{Rational polyominos}

We will construct the cells of the allocations iteratively. To do so, we want
the previously constructed cells to be well-behaved subsets of the plane,
while still allowing sufficient flexibility in the construction of new cells.
The following definition strikes the appropriate balance.

A \df{rational polyomino} is a union of finitely many closed rational
rectangles of the form
\[ [a,b]\times[c,d] \subset \R^2,\quad a,b,c,d\in\Q. \]
By taking the least common denominator, a rational polyomino can also be
expressed as a union of squares
\begin{equation}
\frac{1}{m}\bigcup_{z\in S} \bigl(z+[0,1]^2\bigr),
\label{squares}\end{equation} for some positive integer $m$ and some finite
$S\subset\Z^2$. We write $A^o$ for the topological interior of a set
$A\subseteq \R^2$, and $\overline{A}$ for the closure.
 We call a rational polyomino \df{simple} if its
interior and its complement are both connected, or
equivalently if both the set $S$ and its complement
$\Z^2\setminus S$ induce connected subgraphs of the
nearest-neighbour lattice $\Z^2$ (the graph in which
vertices $x,y\in\Z^2$ are joined by an edge whenever
$\|x-y\|_1=1$).

The next lemma says that we can find a simple rational polynomial of any
suitable area that contains one given set but avoids others.
See Figure \ref{lemmafig} for an illustration.

\begin{lemma}\label{poly}\sloppypar
Let $A$ be a simple rational polyomino, and let $B$ and
$D_1,\ldots,D_r$ be pairwise disjoint subsets of $A^\circ$,
each of which is either a simple rational polyomino or a
singleton. Then, for any rational $\rho$ with $\leb B <\rho
< \leb(A\setminus \bigcup_i D_i)$, there exists a simple
rational polyomino $C$ with $\leb C=\rho$ and
$$B\subset C \subset A^\circ\setminus \textstyle\bigcup_i
D_i.$$
\end{lemma}

\begin{center}
\begin{figure}[htbp]
\resizebox{\textwidth}{!}{\input{lemmadiagram4.pspdftex}}
 \caption{An illustration of Lemma \ref{poly}. On the left,
a
simple rational polyomino $A$, containing in its interior
simple rational polyominos $B$ and $D_1$, and singletons $D_2$ and $D_3$.
On the right, a simple rational polyomino $C$ (shaded)
within the interior of $A$ that
contains $B$ and avoids $D_1$, $D_2$, $D_3$.
 }
 \label{lemmafig}
\end{figure}
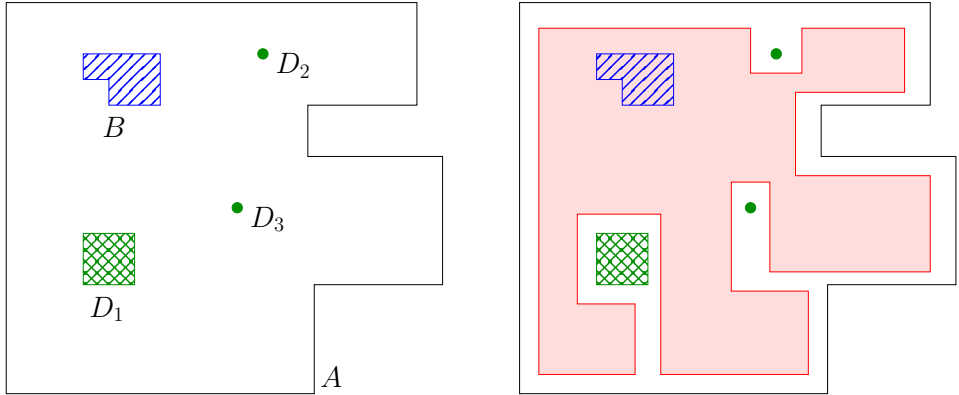
\end{center}

\begin{proof}
We first observe that any singletons among the given sets may be replaced
with simple rational polyominos. Let $k$ be a positive integer, and, for each
singleton set $D_i=\{x_i\}$, let $D'_i$ be the union of all squares of the
form $k^{-1} ([0,1]^2+z)$ for $z\in\Z^2$ that contain the point $x_i$ (at
most $4$ of them).  For non-singleton sets $D_j$ let $D_j'=D_j$.  Similarly
define $B'$ in terms of $B$.  For $k$ sufficiently large, $D'_1,\dots,D'_r$
and $B'$ are pairwise disjoint subsets of $A^\circ$, and $\leb B' <\rho <
\leb(A\setminus \bigcup_i D'_i)$. Therefore, it suffices to prove the lemma
in the case when there are no singletons.

There exists an integer $m$ such that each of the
polyominos $A$, $B$, and $D_1,\dots, D_r$ can be expressed
as a union of squares of side $1/m$ as in \eqref{squares}.
Thus, let $K, L\subset \Z^2$ be such that
\begin{equation}
\label{Krep}
\overline{A\setminus {\textstyle\bigcup_i} D_i \setminus B}
=
\frac{1}{m}\bigcup_{z\in K} \bigl(z+[0,1]^2\bigr);
\qquad
B
=
\frac1m\bigcup_{z\in L}\bigl(z+[0,1]^2\bigr).
\end{equation}
Note that $K$ and $L$ are disjoint. Both $L$ and its
complement are connected (as subsets of $\Z^2$), while $K$
is connected but its complement need not be.

We now further subdivide the squares in (\ref{Krep}). Given
rational $s,t\in(0,1)$ and $z\in\Z^2$, consider the
rectangle of area $st$ within $z+[0,1]^2$ given by
\[
Q_z^{s,t}:=z+\bigl[\tfrac12-\tfrac{s}{2},\tfrac12+\tfrac{s}{2}\bigr]
\times
\bigl[\tfrac12-\tfrac{t}{2},\tfrac12+\tfrac{t}{2}\bigr].
\]

Let $w$ be an element of $L$ that is adjacent in $\Z^2$ to
some element of $K$.  This is possible because $K\cup L$
corresponds to $\overline{A\setminus \bigcup_i D_i}$ and is
therefore connected.
Now take a spanning tree of the set $K\union\{w\}$ in $\Z^2$. Consider the
set that comprises the rectangle $Q_u^{s,t}$ for each $u\in K$, together with
the rectangle that is the convex hull of $Q_u^{s,t}\union Q_v^{s,t}$ for each
edge $(u,v)$ of the spanning tree.  Take the union of this set with $B$, and
call it $C^{s,t}$.

The set $m^{-1} C^{s,t}$ is a simple rational polyomino that contains $B$ and
is contained in $A^\circ\setminus\Union D_i$. To complete the proof, we will
show that we can choose rational $s,t\in(0,1)$ so that $\leb (m^{-1}
C^{s,t})=\rho$, which is to say $\leb C^{s,t}=m^2\rho$.

Note that $\leb C^{s,t}$ can be expressed as the sum of the
following terms: $\leb B$, plus $st$ for each element of
$K$, plus $s(1-t)/2$ for each horizontal edge of the tree
that is incident to $w$, and $s(1-t)$ for each other
horizontal edge of the tree, plus similarly $(1-s)t/2$ or
$(1-s)t$ for each vertical edge. Therefore,
\begin{equation}\label{Csteq}
\leb C^{s,t}=\alpha st + \beta s + \gamma t + \delta
\end{equation}
for some rational $\alpha,\beta,\gamma,\delta$ that do not
depend on $s,t$.  Moreover, since the number of edges of
the tree equals the number of elements of $K$, and at least
one edge is incident to $w$, we have $\alpha, \delta>0$ and
$\beta, \gamma\geq 0$.

The expression in \eqref{Csteq} is continuous and strictly increasing in both
$s$ and $t$.  As $(s,t)\to (0,0)$ we have $m^{-2}\leb C^{s,t}\to \leb
B<\rho$, while as $(s,t)\to(1,1)$ we have $m^{-2}\leb C^{s,t}\to
\leb(A\setminus\Union D_i)
>\rho$.  Hence, writing $s_0=\sup\{s: m^{-2}\leb C^{s,0}<\rho\}$ and
$s_1=\inf\{s:m^{-2}\leb C^{s,1}>\rho\}$, we have $s_0<s_1$.
Fix a rational $s\in(s_0,s_1)$; we have $s\in(0,1)$ and
$m^{-2}\leb C^{s,0}<\rho<m^{-2}\leb C^{s,1}$.  Thus there
exists $t\in(0,1)$ with $m^{-2}\leb C^{s,t}=\rho$; by
\eqref{Csteq}, this $t$ must be rational.
\end{proof}

\section{Non-touching allocation}
\begin{proof}[Proof of Theorem \ref{alloc}(a)]
We first construct an allocation that is invariant under
all translations by elements of $\Z^2$ and whose cells have
the claimed properties; we will obtain a fully
translation-invariant version by translating both the
allocation and the point process by a uniformly random
element of $[0,1)^2$.

The cells of our allocation will be simple rational
polyominos.  We first define a sequence of successively
coarser partitions of $\R^2$ into squares in a
$\Z^2$-invariant way.  (This construction is standard; see
e.g.\ \cite{hpps}). Let $(\alpha_i)_{i=0,1,\ldots}$ be
i.i.d.\ uniformly random elements of the discrete cube
$\{0,1\}^2$, independent of $\Pi$. Given the sequence
$(\alpha_i)$, define a $k$-\df{block} for $k\geq 0$ to be
any set of the form $[0,2^k)^2+z2^k+\sum_{i=0}^{k-1}
\alpha_i 2^i$, for $z\in\Z^2$. (So a $(k+1)$-block is the
disjoint union of four $k$-blocks, and every $k$-block has
area $4^k$.)

We now construct an allocation in a sequence of steps
$k=1,2,\ldots$. At step $k$ we will construct some cells,
each of which will be confined within the interior of some
$k$-block. For step $1$ we proceed as follows.  For each
$1$-block $R$, let $x_1,x_2,\ldots, x_s$ be the points of
$[\Pi]\cap R$, enumerated lexicographically, say. If $s\geq
1$, let $C_1$ be a rational polyomino of area $1$ that
satisfies $x_1\in C_1\subset R^\circ$ and that avoids the
other points $x_2,\ldots,x_s$; this exists by
Lemma~\ref{poly}, with $A=\overline{R}$.  Declare $C_1$ be
the cell allocated to the point $x_1$.  Now if $s\geq 2$,
similarly find a rational polyomino $C_2$ of area $1$ in
$R$ that contains $x_2$ and avoids $C_1$ and
$x_3,\ldots,x_s$, and allocate it to $x_2$. Similarly if
$s\geq 3$, allocate to $x_3$ a cell avoiding $C_1\cup C_2$
and $x_4,\ldots,x_s$.  In each case, this is possible by
Lemma~\ref{poly}, because the total area required for
$C_1,C_2,C_3$ is $3$, which is strictly less than $\leb
R=4$.

For step $k$ we proceed as follows.  Let $R$ be a
$k$-block, and enumerate the unallocated points of
$[\Pi]\cap R$ lexicographically.  For each in turn, use
Lemma~\ref{poly} to choose a rational polyomino of area $1$
in $R^\circ$ that contains the point, and avoids all other
points of $[\Pi]\cap R$ and all previously chosen cells
that intersect $R$ (all such cells are in fact subsets of
$R$). Continue until either we run out of unallocated
points in $R$, or the total area of all the cells in $R$
reaches $\leb R-1$.  Do this for each $k$-block.

After all steps have been completed as above, define an
allocation $\Psi$ by setting $\Psi(y)=x$ if $y$ is in the
cell assigned to $x\in[\Pi]$, and $\Psi(y)=\infty$ for all
other $y\in\R^2$. It is clear that each cell of $\Psi$ is
either empty or a simple rational polyomino of area $1$
that contains the corresponding point of $\Pi$.  It is also
clear that $\Psi$ has the required $\Z^2$-invariance
property provided the cells are chosen according to fixed
translation-invariant rules; this is possible since all the
steps in the proof of Lemma~\ref{poly} can be carried out
in a translation-invariant way.  Every cell is a closed
set, and hence any two non-empty cells are at positive
distance from each other, since they do not intersect.

Now let $U$ be a uniformly random element of the unit
square $[0,1)^2$, independent of $(\Pi,\Psi)$, and define a
translated allocation $\Psi'$ by $\Psi(x):=U+\Psi(x-U)$.
Then $\Psi'$ is a fully translation-invariant allocation of
the translated point process $\Pi'$ defined by
$\Pi'(A):=\Pi(A+U)$ (which is a Poisson process). It
remains to show that every point of the process is
allocated a non-empty cell, and that almost every
$x\in\R^2$ is allocated to some cell, and that the claimed
diameter bound holds.

Let $D$ be the diameter of the cell of $\Psi'$ containing
the origin $0$, if it exists, and let $D=\infty$ if
$\Psi'(0)=\infty$. Then $D$ has the same law as the
diameter of the cell of $\Psi$ containing a uniformly
random point $U$ in $[0,1)^2$.  Note that any cell that is
constructed at step $k$ or earlier lies entirely within
some $k$-block, and therefore has diameter at most
$2^k\surd 2$.  By $\Z^2$-invariance, the probability that
$U$ is allocated by step $k$ equals the expected proportion
of the $k$-block containing $[0,1)^2$ that is allocated by
step $k$. Since the positions of blocks are independent of
$\Pi$, this expected proportion remains the same if we
condition the $k$-block to have a specific position, say
$S:=[0,2^k)^2$. The total area allocated within $S$ by step
$k$ is precisely $\min\{\Pi(S),4^k-1\}$ (since new cells
are added while there are unallocated points until their
total area is one less than the area $4^k$ of $S$). Thus
for all integers $k\geq 1$,
\begin{align*}
\P\bigl(D>2^k\surd 2\bigr)
&\leq
 1- 4^{-k}\,\E\min\bigl\{\Pi(S),4^k-1\bigr\} \\
&\leq 4^{-k}\,\Bigl(1+\E \bigl[(4^k-\Pi(S))^+\bigr]\Bigr)
\end{align*}
Since $\Pi(S)$ is Poisson distributed with mean $4^k$, we have $\E
[(4^k-\Pi(S))^+]\leq C\sqrt{4^k}$ for some $C$, and it follows that
$\P(D>t)<c/t$ as claimed.

In particular the above implies that $D<\infty$ almost
surely, and so almost every $x\in\R^2$ is assigned to some
cell by $\Psi'$.  Since each cell has area $1$, a standard
mass-transport argument (see e.g.\ \cite{hhp,hpps}) then
implies that the process of those points of $\Pi'$ that are
allocated cells has intensity $1$.  Since $\Pi'$ has
intensity $1$, this shows that almost surely every point of
$\Pi'$ is allocated.
\end{proof}

\section{Locally finite allocation}
\begin{proof}[Proof of Theorem \ref{alloc}(b)]
As at the beginning of the proof of part (a), we define a
hierarchy of $k$-blocks using an i.i.d.\ sequence
$(\alpha_i)$.  As in the previous proof, it suffices to
construct an appropriate allocation that is invariant under
$\Z^2$, and then apply a random translation.

For each block we define an inner block. Let
$(\eta_k)_{k\geq 1}$ be a strictly decreasing sequence of
rational numbers in $(\tfrac12,1)$ with $\eta_k \downarrow
\tfrac12$ as $k\to\infty$. If $B_k$ is a $k$-block then
$B_k=(a,b)+[0,2^k)^2$ for some point $(a,b)\in\Z^2$. Define
its \df{inner block} $I_k$ by $I_k=(a,b)+[\eta_k,
2^k-\eta_k]^2$. Thus $I_k$ is a square of side
$2^k-2\eta_k$ with the same centre as $B_k$.  Define also
$M_k=B_k\intersect I_{k+1}$, where $I_{k+1}$ is the inner
block of the $(k+1)$-block containing $B_k$. Thus $M_k$ is
a square of side $2^k-\eta_{k+1}$, which contains $I_k$ in
its interior (since the sequence $\eta_k$ is strictly
decreasing).

As in the previous proof, we construct the allocation in a
sequence of steps $k=1,2,\dots$. At step $k$ we add some
cells to the allocation, with each such cell confined to
the interior of some $k$-block.

At step $k$ we treat each $k$-block separately. Let $B_k$
be a $k$-block. The following statement plays the role of
induction hypothesis: at the beginning of step $k$, the
closure of the union of the previously allocated cells in
$B_k$ is a union of disjoint simple rational polyominos
contained in the interior of $I_k$. In particular, the
complement with respect to $I_k$ of this set is connected.

The allocations during step $k$ will be carried out in such a way that
at the end of step $k$, the following holds:
the closure of the union of the allocated cells in $B_k$ forms a collection
of disjoint simple rational polyominos contained in $M_k^o$.
This property implies the induction hypothesis for next level $k+1$;
for if $B_{k+1}$ is the $(k+1)$-block containing $B_k$,
then its inner block $I_{k+1}$ is made up of the square $M_k$
together with three other analogous squares; the interiors of these squares are disjoint.

Now we explain how the allocation within $B_k$ at step $k$ is carried out.
We say that the box $B_k$ is \textbf{good} if
\begin{equation}\label{good}
\leb I_k < \Pi(I_k) < \leb M_k.
\end{equation}
We proceed in two different ways depending on whether or not $B_k$ is good.

If $B_k$ is not good, we allocate only within the inner box
$I_k$. (By the induction hypothesis, all previously
allocated cells in $B_k$ are in the interior of $I_k$). In
this case we proceed in the same fashion as we did in the
construction of the non-touching allocation in part (a).
Using Lemma~\ref{poly}, we add new cells in the interior of
$I_k$ one by one, each one a simple rational polyomino
disjoint from previous cells. We continue until either the
number of cells is $\Pi(I_k)$ or the remaining unallocated
area inside $I_k$ is at most 1.

If $B_k$ is good, we start by finding a region lying
between $I_k$ and $M_k$ which contains a number of points
of $\Pi$ exactly equal to its area. Because (\ref{good})
holds, this can be done using Lemma~\ref{poly}, setting
$\rho=\Pi(I_k)$, $B=I_k$ and $A=M_k$, and setting
$D_1,\ldots,D_r$ to be the points of $\Pi$ in $M_k\setminus
I_k$ (and $r=0$). In this way we find a simple rational
polyomino $C$ with $I_k\subset C \subset M_k^o$ and
$\Pi(C)=\leb(C)$.

Now we will divide up the set $C$ to form the cells
allocated to points of $\Pi$ in $C$. Some such allocations
may already have been done. All the remaining ones except
the last can be done one by one, just as before, using
Lemma~\ref{poly}. These new allocations are simple random
polyominos, disjoint from previous cells, containing
precisely one point of $\Pi$ and contained in the interior
of $C$; in particular the remainder of $C$ stays connected.
Finally, when one point of $\Pi$ remains in $C$, and hence
when area 1 remains to be allocated, we allocate the rest
of $C$ as the cell of the last point. The closure of this
cell is a rational polyomino, and is connected but not
simple.

At the end of the procedure, as in part (a) define $\Phi$
by setting $\Phi(y)=x$ whenever $y$ is in the cell assigned
to $x\in[\Pi]$, and $\Phi(y)=\infty$ otherwise. Each such
cell is either empty or has area 1 and contains the
corresponding point of $\Pi$. As before, by carrying out
the steps of Lemma~\ref{poly} in a translation-invariant
way, we can ensure that $\Phi$ has the required
$\Z^2$-invariance property.

If $B_k$ is a good box, then the number of cells that
intersect $I_k$ is finite. Also, every point in $I_k$ is
allocated to some cell. To show that every point in $\R^2$
is allocated to some cell and that the allocation is
locally finite as desired, it will be enough to show that
with probability 1, every point is in the interior of the
inner box of some good box.

Let $S$ be any $1\times1$ square. The probability that $B$
is contained in the interior of the inner box of some
$k$-block is $(2^k-2\eta_k-1)^2/4^k$, which tends to 1 as
$k\to\infty$. If this event holds for some $k$, then in
fact it holds for all $k'>k$ also (since the inner box of a
$k$-block lies within the inner box of the containing
$(k+1)$-block). Hence with probability 1, this event holds
for all large enough $k$, say $k\geq k_0$.

Now let $B_k$ be the $k$-block containing $S$, with $I_k$
and $M_k$ defined as before. It will be enough to show that
with probability 1, $B_k$ is good for infinitely many $k$.
From (2), $\P(B_k \text{ is good})=\P(\leb I_k < \Pi(I_k) <
\leb M_k)$. We have $\leb I_k = (2^k-2\eta_k)^2$, and
$\leb M_k = (2^k-\eta_{k+1})^2$.  Therefore, since
$\eta_k>\eta_{k+1}>\tfrac12$,
\begin{align*}
 \leb M_k-\leb I_k &= (2^{k+1}
-2\eta_k-\eta_{k+1})(2\eta_k-\eta_{k+1})
>(2^{k+1}-3\eta_k)\tfrac12\\ &> \sqrt{\leb I_k}.
\end{align*}

Since $\Pi(I_k)$ is  $\text{Poisson}(\leb I_k)$ and $\leb
I_k\to\infty$, we obtain $\P(B_k \text{ is good})\geq c$
for all large enough $k$ for some constant $c$ (in fact,
any $c<\P(0<Z<1)$ is enough, where $Z$ is a standard
Gaussian).

If the events $\{B_k \text{ is good}\}$ were independent
for different $k$, this would be enough; however, we need
to control the dependence. To do this, consider any
sequence $k_1<k_2<\dots$ with the following properties:
\begin{itemize}
\item[(i)] $\P\Bigl(\Pi\left(I_{k_n}\right) > \frac12
    \sqrt{\leb I_{k_{n+1}}}\Bigr) < 2^{-n}$;
\item[(ii)] $\P\Bigl( \leb I_{k_{n+1}} <
    \Pi\left(I_{k_{n+1}}\setminus I_{k_n}\right) < \leb
    I_{k_{n+1}} + \frac12 \sqrt{\leb I_{k_{n+1}}} \Bigr)
    > c'$ \\ (where $c'$ is a constant independent of $k$).
\end{itemize}
By similar arguments to the above, this is easily shown to
be possible by making $k_n$ grow quickly enough. The events
in (ii) are independent for different $n$, so by
Borel-Cantelli, with probability 1 infinitely many of them
occur. The sum of the probabilites in (i) is finite, so
with probability 1 only finitely many of them occur. But
for any given $n$, if the event in (i) fails and the event
in (ii) holds then $B_{k_{n+1}}$ is good. So with
probability 1, there are infinitely many $k$ for which
$B_k$ is good, as desired.

Finally, we turn to the diameter bound. At step $k$, the
area allocated within a $k$-block $B_k$ is at least
$\min\left(\Pi(I_k)-1, \leb I_k-2\right)$. Arguing as for
the non-touching allocation in part (a), we obtain
\[
\P(D>2^k\sqrt{2}) \leq 1-4^{-k} \E \min\left( \Pi(I_k)-1, \leb I_k-2\right).
\]
where $D$ is the diameter of the cell containing the
origin, in the allocation obtained by translating $\Phi$ by
a random element of $[0,1)^2$. As before, this is easily
seen to be at most $C\sqrt{4^k}$ for some $C<\infty$,
giving the desired bound on the tail of $D$.
\end{proof}

\bibliographystyle{abbrv}
\bibliography{alloc}

\end{document}

%% file: lemmadiagram4.pspdftex
\begin{picture}(0,0)%
\includegraphics{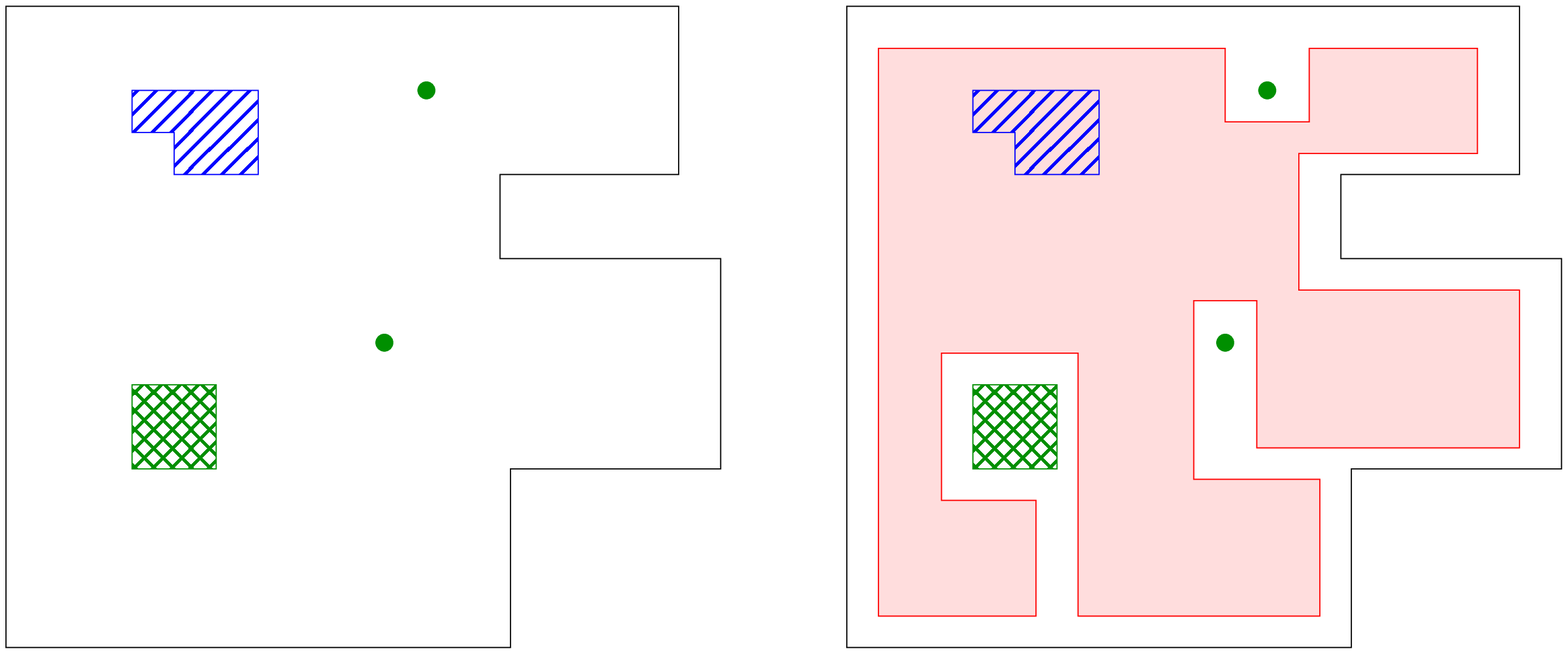}%
\end{picture}%
\setlength{\unitlength}{3947sp}%
\begingroup\makeatletter\ifx\SetFigFont\undefined%
\gdef\SetFigFont#1#2#3#4#5{%
  \reset@font\fontsize{#1}{#2pt}%
  \fontfamily{#3}\fontseries{#4}\fontshape{#5}%
  \selectfont}%
\fi\endgroup%
\begin{picture}(11124,4635)(889,-6184)
\put(1876,-5236){\makebox(0,0)[lb]{\smash{{\SetFigFont{20}{24.0}{\rmdefault}{\mddefault}{\updefault}{\color[rgb]{0,0,0}$D_1$}%
}}}}
\put(2026,-3136){\makebox(0,0)[lb]{\smash{{\SetFigFont{20}{24.0}{\rmdefault}{\mddefault}{\updefault}{\color[rgb]{0,0,0}$B$}%
}}}}
\put(3751,-4186){\makebox(0,0)[lb]{\smash{{\SetFigFont{20}{24.0}{\rmdefault}{\mddefault}{\updefault}{\color[rgb]{0,0,0}$D_3$}%
}}}}
\put(4051,-2386){\makebox(0,0)[lb]{\smash{{\SetFigFont{20}{24.0}{\rmdefault}{\mddefault}{\updefault}{\color[rgb]{0,0,0}$D_2$}%
}}}}
\put(4576,-6061){\makebox(0,0)[lb]{\smash{{\SetFigFont{20}{24.0}{\rmdefault}{\mddefault}{\updefault}{\color[rgb]{0,0,0}$A$}%
}}}}
\end{picture}%